\documentclass[12pt]{amsart}

\usepackage{amssymb,latexsym,amscd}
\usepackage{amsthm,amsmath,mathrsfs}
\usepackage{youngtab}

\newcommand{\cV}{\mathcal{V}}
\newcommand{\cH}{\mathcal{H}}
\newcommand{\cU}{\mathcal{U}}

\newcommand{\cO}{\mathcal{O}}

\newcommand{\cF}{\mathcal{F}}

\newcommand{\cB}{\mathcal{B}}
\newcommand{\cD}{\mathcal{D}}

\newcommand{\cL}{\mathcal{L}}

\newcommand{\cC}{\mathcal{C}}

\newcommand{\cS}{\mathcal{S}}

\newcommand{\lra}{\longrightarrow}

\newcommand{\ra}{\rightarrow}

\newcommand{\PP}{\mathbb{P}}

\newcommand{\YY}{\mathcal{Y}}

\newcommand{\ZZ}{\mathbb{Z}}

\newcommand{\CC}{\mathbb{C}}

\newcommand{\slfr}{\mathfrak{sl}}

\theoremstyle{plain}

\newtheorem{thm}{Theorem}[section]

\newtheorem{lem}[thm]{Lemma}

\newtheorem{prop}[thm]{Proposition}

\newtheorem{cor}[thm]{Corollary}

\newtheorem{rem}[thm]{Remark}

\newtheorem{defi}[thm]{Definition}

\newtheorem{ex}[thm]{Example}

\let\mathbb=\mathbf

\begin{document}

\title[]{Strange duality revisited}

\begin{abstract}
We give a proof of the strange duality or rank-level duality of the WZW models of conformal
blocks by extending the genus-$0$ result, obtained by Nakanishi-Tsuchiya in 1992, to higher genus curves
via the sewing procedure. The new ingredient of the proof is an explicit use of the branching rules of the 
conformal embedding of affine Lie algebras $\widehat{\slfr(r)} \times \widehat{\slfr(l)} \subset
\widehat{\slfr(rl)}$. We recover the strange duality of spaces of generalized theta functions obtained by
Belkale, Marian-Oprea, as well as by Oudompheng in the parabolic case. 
\end{abstract}

\author{Christian Pauly}

\address{Laboratoire de Math\'{e}matiques J. A. Dieudonn\'e \\ Universit\'e de Nice - Sophia Antipolis \\
06108 Nice Cedex 02 \\ France}

\email{pauly@unice.fr}

\thanks{This research was supported by a Marie Curie Intra European Fellowship within the 7th European Community
Framework Programme}



\maketitle

\bigskip

\section{Introduction}

One of the most significant recent results in the theory of vector bundles over curves is the proof of the
strange duality or rank-level duality given by Belkale and Marian-Oprea. For a survey of the results
we refer e.g. to \cite{MO2}, \cite{Pa2}, \cite{Po}.  In this note we give another proof of that duality
in the framework of conformal blocks. In fact, the duality statement for conformal blocks over the
projective line was proved by Nakanishi-Tsuchiya in 1992 \cite{NT}. We generalize their
statement to a smooth projective curve of any genus. 

\bigskip

In order to state the Main Theorem we need to introduce some notation. Let $r,l \geq 2$ be integers. We denote
by $P_l(r)$ the finite set of dominant weights at level $l$ of the Lie algebra $\slfr(r)$ and for
$\lambda \in P_l(r)$ we denote by  
$\cH_{\lambda,l}(r)$ the irreducible integrable $\widehat{\slfr(r)}$-module of weight $\lambda$ and
level $l$, where $\widehat{\slfr(r)}$ is the affine Lie algebra associated to $\slfr(r)$.
The new ingredient of the proof is the use of the branching rules for the conformal embedding
$\widehat{\slfr(r)} \times \widehat{\slfr(l)} \subset \widehat{\slfr(rl)}$ of the irreducible integrable 
$\widehat{\slfr(rl)}$-module $\cH_{\lambda,1}(rl)$, when $\lambda$ is a (possibly zero) fundamental 
dominant weight $\slfr(rl)$. The branching rule \cite{H} gives the decomposition 
$$ \cH_{\lambda,1}(rl) = \bigoplus_{ Y \in \YY_{r,l}^{aff}(\lambda)} \cH_{\mu,l}(r) \otimes 
\cH_{{}^t\mu,r}(l), $$
where $Y$ varies over a finite set  $\YY_{r,l}^{aff}(\lambda) \subset \YY_{r,l}^{aff}$ of Young diagrams of 
type $(r,l)$ and of size $\lambda$ (for the definitions see section \ref{Young}). The dominant weights 
$\mu$ and ${}^t\mu$ of $\slfr(r)$ and $\slfr(l)$ are naturally associated to $Y$ and to its transpose ${}^tY$.
The above decomposition is an infinite-dimensional analogue of the classical Skew Cauchy Formula 
(see e.g. \cite{Pr}) giving
the branching rule of the representation $\Lambda^\lambda \left( \CC^r \otimes \CC^l \right)$ 
for the embedding of finite Lie algebras 
$\slfr(r) \times \slfr(l) \subset \slfr(rl)$. Given a smooth projective curve $C$ with $n$ marked points
and a collection $\vec{\lambda} = (\lambda_1, \ldots, \lambda_n) \in P_l(r)^n$ we denote by 
$\cV^\dagger_{\vec{\lambda},l} (C,r)$ the corresponding conformal block. With this notation
we now can state the

\bigskip

{\bf Main Theorem.} {\em  Let $C$ be a smooth projective complex curve with $n$ marked points and let
$$\vec{\lambda} = (\lambda_1, \ldots, \lambda_n) \in P_1(rl)^n = \{ 0, 1,  \ldots, rl-1 \}^n$$ 
be a labelling of the marked points with fundamental weights of $\slfr(rl)$ satisfying the condition
$\sum_{i=1}^n \lambda_i \equiv 0 \  \mathrm{mod} \ rl$. 
For any collection of Young diagrams
$$\vec{Y} = (Y_1, \ldots, Y_n) \in \prod_{i=1}^n \YY_{r,l}^{aff}(\lambda_i)$$ 
we denote by $\vec{\mu} = \pi (\vec{Y}) \in P_l(r)^n$ and
${}^t\vec{\mu} = \pi ({}^t\vec{Y}) \in P_r(l)^n$ the collections of associated 
dominant weights of $\slfr(r)$ and $\slfr(l)$ respectively. Then the natural linear map
between spaces of conformal blocks over the pointed curve $C$ obtained via the conformal embedding
$\widehat{\slfr(r)} \times \widehat{\slfr(l)} \subset
\widehat{\slfr(rl)}$

$$ \alpha:  \cV^\dagger_{\vec{\lambda},1} (C, rl) \lra
\cV^\dagger_{\vec{\mu},l} (C, r) \otimes \cV^\dagger_{{}^t \vec{\mu},r} (C, l) $$
induces an {\em injective} linear map
$$ SD_{\vec{Y}} : \cV_{\vec{\mu},l} (C,r) \lra
\cV_{\vec{\lambda},1} (C, rl) \otimes \cV^\dagger_{{}^t \vec{\mu},r} (C, l). $$
}

\bigskip

We recall that there is a canonical isomorphism (up to homothety) between the space of conformal
blocks associated to a $1$-pointed curve labelled with the trivial weight $\lambda_1 = 0$
\begin{equation} \label{cbiso}
\cV^\dagger_{0,l}(C,r) \cong H^0(\mathcal{SU}_C(r), \cL^l) 
\end{equation}
and the so-called space of generalized theta functions of rank $r$ and level $l$, i.e., the space
of global sections of the $l$-th power of the determinant line bundle $\cL$ over the
coarse moduli space $\mathcal{SU}_C(r)$ of semi-stable rank-$r$ vector 
bundles with fixed trivial determinant over the
curve $C$. In the special case of a $1$-pointed curve labelled with the trivial weight and
$Y_1 = 0, \mu_1 = 0, {}^t\mu_1 = 0$ the above theorem combined with the isomorphism
\eqref{cbiso} states that the linear map
$$ SD_0 : H^0(\mathcal{SU}_C(r), \cL^l)^\dagger  \lra 
H^0(\mathcal{SU}_C(rl), \cL)^\dagger  \otimes H^0(\mathcal{SU}_C(l), \cL^r) $$
is injective.  We denote by $\cU_C^*(r)$ the coarse moduli space of semi-stable rank $r$ and degree $r(g-1)$
vector bundles over $C$ and by $\Theta$ the divisor $\{ E \in \cU_C^*(r) \ | \ \dim H^0( C,E) \not= 0 \} 
\subset \cU_C^*(r)$. The tensor product map
$\cU_C^*(1) \times \mathcal{SU}_C(l) \ra \cU_C^*(l)$ induces an inclusion
$$H^0(\cU_C^*(l), \cO(r\Theta)) \subset H^0(\cU_C^*(1), \cO(rl \Theta)) \otimes 
H^0(\mathcal{SU}_C(l), \cL^r) \cong 
H^0(\mathcal{SU}_C(rl), \cL)^\dagger  \otimes H^0(\mathcal{SU}_C(l), \cL^r).$$
The last isomorphism is proved in \cite{BNR}. It is shown in \cite{Be2} that the image of the
linear map $SD_0$ is contained in $H^0(\cU_C^*(l), \cO(r\Theta))$. Hence, assuming the well-known fact that
both vector spaces have the same dimension, we obtain a new proof
of the following theorem.

\begin{thm} [\cite{Be1}, \cite{MO}]
For any smooth curve $C$, the linear map
$$SD_0 : H^0(\mathcal{SU}_C(r), \cL^l)^\dagger  \lra H^0(\cU_C^*(l), \cO(r\Theta))$$
is an isomorphim.
\end{thm}

Note the map $SD_0$ is defined here in terms of conformal blocks, but it coincides 
(see \cite{Be2} Proposition 5.2) 
under the isomorphism
\eqref{cbiso} with the one defined at the level of moduli spaces of vector bundles.

\bigskip

The parabolic version of Theorem 1.1 proved by Oudompheng \cite{O} Theorem 4.10 can be
similarly deduced from our Main Theorem by using the parabolic version of the isomorphism
\eqref{cbiso} proved in \cite{Pa1}. We leave the details to the reader.

\bigskip

The paper is organised as follows. In sections 2,3 and 4 we collect for the reader's convenience
some known results on conformal blocks, on the branching rules and on the WZW-connection. The proof
of the Main Theorem is given in section 5.

\bigskip

I would like to thank Laurent Manivel for helpful comments.

\bigskip

\section{Conformal blocks and factorization}

\subsection{Definition and properties of conformal blocks}

Given an integer $l \geq 1$ called the level, we introduce the finite set of dominant weights of $\slfr(r)$
$$P_l(r) = \{ \lambda = \sum_{i=1}^{r-1} a_i \varpi_i \ | \ \sum_{i=1}^{r-1} a_i \leq l ; \ a_i \geq 0 \}, $$
where the $\varpi_i$ denote the $r-1$ fundamental weights of $\slfr(r)$. We also consider the involution 
of the set $P_l(r)$  
\begin{equation} \label{invodual}
\lambda = \sum_{i=1}^{r-1} a_i \varpi_i \ \mapsto \ 
\lambda^\dagger  = \sum_{i=1}^{r-1} a_{r-i} \varpi_i.
\end{equation}

We denote by $V_\lambda$ the irreducible $\slfr(r)$-module with dominant weight $\lambda$. Then $\lambda^\dagger$
is the dominant weight of the dual $V_\lambda^\dagger$.

\bigskip

For the level $l=1$ we will often identify $P_1(r)$ with the set of integers $\{ 0 , 1, \ldots , r-1 \}$
mapping the trivial weight to $0$ and the $i$-th fundamental weight $\varpi_i$ to $i$. Under this
identification the above involution \eqref{invodual} becomes $0^\dagger = 0$ and 
$i^\dagger = r-i$ for $i > 0$.

\bigskip

Given an integer $n \geq 1$, a collection $\vec{\lambda} =  (\lambda_1, \ldots ,
\lambda_n) \in P_l(r)^n$ of dominant weights of $\slfr(r)$ and a family
\begin{equation} \label{definitionfamilyF}
\cF = ( \pi: \cC \ra \cB ; s_1, \ldots , s_n ; \xi_1 , \ldots , \xi_n) 
\end{equation}
of $n$-pointed stable curves of arithmetic genus $g$ parameterized by a base variety $\cB$ with sections 
$s_i : \cB \ra \cC$ and
formal coordinates $\xi_i$ at the divisor $s_i(\cB) \subset \cC$, one constructs (see \cite{TUY} section 4.1)
a locally free sheaf
$$\cV^\dagger_{\vec{\lambda},l}(\cF,r)$$
over the base variety $\cB$, called
the {\em sheaf of conformal blocks} or the {\em sheaf of vacua} for the Lie algebra $\slfr(r)$ and the markings
$\vec{\lambda}$ at level $l$. We recall that
$\cV^\dagger_{\vec{\lambda},l}(\cF,r)$ is a subsheaf of $\cO_{\cB} \otimes
\cH^\dagger_{\vec{\lambda},l}$, where $\cH^\dagger_{\vec{\lambda},l}$ denotes the dual
of the tensor product 
$$\cH_{\vec{\lambda},l} = \cH_{\lambda_1,l} \otimes \cdots \otimes
\cH_{\lambda_n,l}$$ 
of the integrable highest weight representations $\cH_{\lambda_i,l}$ of
level $l$ and weight $\lambda_i$ of the affine Lie algebra $\widehat{\slfr(r)}$.
The formation of the sheaf of
conformal blocks commutes with base change. In particular, we have for any point $b \in \cB$
$$  \cV^\dagger_{\vec{\lambda}, l}(\cF,r) \otimes_{\cO_{\cB}} \cO_b \cong  \cV^\dagger_{\vec{\lambda},l}(C,r), $$
where $C$ denotes the data $(\cC_b = \pi^{-1}(b); s_1(b), \ldots , s_n(b); \xi_{1|\cC_b}, \ldots , \xi_{n|\cC_b})$
consisting of a stable curve $\cC_b$ with $n$ marked points $s_1(b), \ldots, s_n(b)$ and formal coordinates
$\xi_{i|\cC_b}$ at the points $s_i(b)$.

We recall that the sheaf of conformal blocks $ \cV^\dagger_{\vec{\lambda},l}(\cF, r)$  does not depend (up to a canonical isomorphism) on the formal coordinates $\xi_i$ (see e.g. \cite{U} Theorem 4.1.7). We therefore omit the formal 
coordinates in the notation.

\bigskip

We have the following factorization theorem.

\begin{prop}[\cite{TUY}] \label{factorization}
Let $C$ be a nodal curve with a node $n$ and let $\pi: \widetilde{C} \ra C$ be the partial desingularization
at $n$. Then we have the direct sum decomposition
$$  \cV^\dagger_{\vec{\lambda},l}(C, r) = \bigoplus_{\mu \in P_l(r)} 
\cV^\dagger_{\vec{\lambda} \cup \mu \cup \mu^\dagger,l}(\widetilde{C}, r), $$
where we put the weights $\mu$ and $\mu^\dagger$ at the two points $a,b \in \widetilde{C}$ lying over the
node $n \in C$.
\end{prop}

\bigskip

\begin{lem}  \label{dimcblevelone}
Let $C$ be a stable curve of genus $g$ with $n$ marked points labelled with the dominant weights
$\vec{\lambda} = (\lambda_1, \ldots, \lambda_n) \in P_1(rl)^n$. Then
$$ \dim \cV^\dagger_{\vec{\lambda},1}(C,rl) = \left\lbrace \begin{array}{cl} (rl)^g &  \text{if} \ 
\sum_{i=1}^n \lambda_i \equiv 0 \  \mathrm{mod} \ rl  \\ 0 & \text{otherwise} \end{array} \right. $$ 
\end{lem}

\begin{proof}
This is a straightforward consequence of the factorization of conformal blocks when degenerating the 
genus-$g$ curve to a rational curve with $g$ nodes and taking the desingularization --- we iterate Proposition
\ref{factorization} $g$ times. The fomula states that
$$ \dim \cV^\dagger_{\vec{\lambda},1}(C,rl) = \sum_{\vec{\mu} \in P_1(rl)^g}
\dim \cV^\dagger_{\vec{\lambda} \cup \vec{\mu} \cup \vec{\mu}^\dagger, 1} (\PP^1,rl)$$
On the other hand $\sum_{i=1}^n \lambda_i + \sum_{j=1}^{g} \mu_j + \mu_j^\dagger \equiv 
\sum_{i=1}^n \lambda_i \  \mathrm{mod} \ rl$, since $\mu_j + \mu_j^\dagger \equiv 0 \ 
\mathrm{mod} \  rl$. Thus, it is 
sufficient to prove the dimension formula for $g=0$. Then use once more the factorization
formula to reduce to the case of $\PP^1$ with three marked points, i.e. $g=0$ and $n=3$. 
That calculation is standard, see e.g. \cite{G} Formula 2.14 or
\cite{NT} Section 4.
\end{proof}

\section{Branching rules}

In this section we review the results on the branching rules of \cite{H} used in the proof of the Main Theorem.

\subsection{Young diagrams} \label{Young}

Given two positive integers $r$ and $l$, we will denote by Young diagram of type $(r,l)$ a decreasing sequence 
of $r$ positive integers 
$$ Y = ( y_1 \geq y_2 \geq \ldots \geq y_{r-1} \geq y_r \geq 0 ) \ \ \text{such that} \ \ y_1 - y_r \leq l$$
We denote the (infinite) set of Young diagrams of type $(r,l)$ by $\YY_{r,l}$ and consider the map to 
the set of dominant weights of $\slfr(r)$
$$ \pi : \YY_{r,l} \lra P_l(r), \qquad  Y \mapsto \pi(Y) = \sum_{i=1}^{r-1}(y_i - y_{i+1}) \varpi_i. $$

We introduce the finite subsets:

$$  \YY_{r,l}^{aff} =  \{ Y  \in  \YY_{r,l} \ | \ y_r \leq l-1 \} \qquad \text{and} \qquad
 \YY_{r,l}^{fin} = \{ Y \in  \YY_{r,l}^{aff} \ | \ y_1 \leq l \}. $$

Note that all fibers of the map $\pi : \YY_{r,l}^{aff} \lra P_l(r)$ have cardinality $l$.

\bigskip

We will now define several maps between these finite sets of Young diagrams. We can think of a Young diagram
of type $(r,l)$ as a collection of $r$ rows, where we put into the $i$-th row $y_i$ boxes. We distinguish
two cases:

\bigskip

\noindent
{\bf Case 1:} $Y \in  \YY_{r,l}^{fin}$. We have $y_1 \leq l$ and $y_r \leq l-1$. We put 
${}^t Y \in \YY_{l,r}^{fin}$ the Young diagram of size $(l,r)$ defined by taking the transpose of $Y$, i.e.
by putting the $y_i$ boxes into the $i$-th column, and $Y^\dagger \in \YY_{r,l}^{fin}$ the complement (after a 
$180$ degree rotation) of $Y$ in the full rectangle consisting of $r$ rows having $l$ boxes each.

\bigskip

\noindent
{\bf Case 2:} $Y \in  \YY_{r,l}^{aff} \setminus \YY_{r,l}^{fin}$. We have $l+1 \leq y_1 \leq 2l-1$ and
$y_r \leq l-1$. We can write the Young diagram $Y$ as a union of two diagrams $Y_1 \cup Y_2$ with
$Y_i \in \YY_{r,l}^{fin}$, where the first one $Y_1$ has in its $i$-th row $\mathrm{min}(l,y_i)$ boxes and 
the second one $Y_2$ has in its $i$-th row $\mathrm{max}(y_i - l ,0)$ boxes. We then define the
transpose ${}^t Y$ to be the union of the transposes ${}^t Y_1 \cup {}^t Y_2$, where the 
${}^t Y_i$ are defined as in Case 1. Note that the Young diagram fits into two full rectangles
consisting of $r$ rows having $2l$ boxes each. We then define $Y^\dagger \in \YY_{r,l}^{aff}$ to be the
complement (after a $180$ degree rotation) of $Y$ in that double rectangle.

\bigskip

We thus have constructed two involutive bijective maps

\bigskip

\begin{center}
$\begin{array}{rcl}
\YY_{r,l}^{aff} & \lra & \YY_{r,l}^{aff} \\
Y & \mapsto & Y^\dagger
\end{array}$ \ \ and \ \ $\begin{array}{rcl}
\YY_{r,l}^{aff} & \lra & \YY_{l,r}^{aff} \\
Y & \mapsto & {}^t Y
\end{array} $ 
\end{center}

\bigskip

These two maps preserve the subsets $\YY_{r,l}^{fin} \subset \YY_{r,l}^{aff}$ and 
$\YY_{l,r}^{fin} \subset \YY_{l,r}^{aff}$. Note that ${}^t(Y^\dagger) = 
({}^t Y)^\dagger$ and that $\pi(Y^\dagger) = \pi(Y)^\dagger$, where $\pi(Y)^\dagger$ is defined as in \eqref{invodual}. We also note the equalities
$$ |\YY_{r,l}^{aff}| = |\YY_{l,r}^{aff}| = l |P_l(r)| = r |P_r(l)|.$$

\bigskip

Given a Young diagram $Y \in \YY_{r,l}$ we define its size $|Y| = \sum_{i=1}^r y_i \ \mathrm{mod} \ rl$. Moreover
we identify $\ZZ/ rl \ZZ = \{ 0 , \ldots, rl -1 \}$, so that $|Y| \in   \{ 0 , \ldots, rl -1 \}$. It is 
clear from the definition that $|Y^\dagger| + |Y| = rl$ and that $|{}^tY| = |Y|$. We will denote by 
$\YY_{r,l}^{fin}(\lambda)$ and $\YY_{r,l}^{aff}(\lambda)$ the corresponding subsets of Young diagrams of size
$\lambda$.

\bigskip

\begin{ex}
{\em We consider the following example for $r=3$ and $l=4$: $Y = (6,4,3) \in  \YY_{3,4}^{aff} \setminus \YY_{3,4}^{fin}$. Then we have the following table.
$$\begin{array}{cccc}
Y = (6,4,3) & {}^t Y = (4,4,3,2) & Y^\dagger = (5,4,2) & {}^t Y^\dagger = (4,3,2,2) \\
 & & & \\
\yng(6,4,3) & \yng(4,4,3,2) & \yng(5,4,2) & \yng(4,3,2,2) \\
 & & & \\
 \pi(Y) = 2\varpi_1 + \varpi_2 &  \pi(Y^\dagger) = \varpi_1 + 2\varpi_2 & 
\pi({}^t Y) = \varpi_2 + \varpi_3 & \pi({}^t Y^\dagger) = \varpi_1 + \varpi_2
\end{array}
$$
Moreover $|Y|=|{}^t Y| = 1$ and $|Y^\dagger| = |{}^t Y^\dagger| = 11$.}
\end{ex}

\bigskip

\subsection{Finite-dimensional case}

We now recall the classical Skew Cauchy Formula (see e.g. \cite{Pr} Theorem 8.4.1. Chapter 9), which
gives the branching rule of the fundamental $\slfr(rl)$-modules under the embedding 
$\slfr(r) \times \slfr(l) \subset \slfr(rl)$.
Let $\lambda$ be in  $\{ 0 , \ldots, rl -1 \}$. Under the identification $\{ 0 , \ldots, rl -1 \} = P_1(rl)$
the $\slfr(rl)$-module $\Lambda^\lambda \CC^{rl} = \Lambda^\lambda \left( \CC^{r} \otimes \CC^{l} \right)$ corresponds to the fundamental weight $\lambda \in P_1(rl)$ and decomposes as sum of irreducible 
$\slfr(r) \times \slfr(l)$-modules 
\begin{equation} \label{SkewCauchy}
V_\lambda =
\Lambda^\lambda \left( \CC^{r} \otimes \CC^{l} \right) = \bigoplus_{ Y \in \YY_{r,l}^{fin}(\lambda)} V_{\mu}
\otimes V_{{}^t \mu},
\end{equation}
where $V_{\mu}$ and $V_{{}^t \mu}$ denote the $\slfr(r)$ and $\slfr(l)$-modules with dominant
weights $\mu = \pi(Y) \in P_l(r)$ and ${}^t \mu= \pi({}^t Y) \in P_r(l)$. 

\bigskip

\subsection{Infinite-dimensional case}

The analogue of the above Skew Cauchy Formula for the embedding of affine Lie algebras
$\widehat{\slfr(r)} \times \widehat{\slfr(l)} \subset \widehat{\slfr(rl)}$ is worked out in
\cite{H} Theorem 4.2. With the above notation we have the decomposition as sum of 
irreducible 
$\widehat{\slfr(r)} \times \widehat{\slfr(l)}$-modules
$$ \cH_{\lambda,1} = \bigoplus_{ Y \in \YY_{r,l}^{aff}(\lambda)} \cH_{\mu,l} \otimes 
\cH_{{}^t \mu,r}. $$

The Virasoro operator $L_0$ associated to $\slfr(rl)$ induces a decomposition 
(see \cite{TUY} or \cite{U}) into eigenspaces of the  $\widehat{\slfr(rl)}$-module
$$ \cH_{\lambda,1} = \bigoplus_{d=0}^\infty \cH_{\lambda,1}(d) \qquad \text{with} \qquad
\cH_{\lambda,1}(0) = V_\lambda.$$
For every Young diagram $Y \in \YY_{r,l}^{aff}(\lambda)$ we have an inclusion 
\begin{equation} \label{inclH}
 \cH_{\mu,l} \otimes \cH_{{}^t \mu,r} \hookrightarrow  \cH_{\lambda,1}.
\end{equation}
Note that the Virasoro operators $L_0$ associated to the two Lie algebras 
$\slfr(r) \times \slfr(l) \subset \slfr(rl)$ coincide since these Lie algebras form a 
conformal pair (see \cite{KW} Proposition 3.2 (c)). Hence restricting the previous
inclusion to the $0$-eigenspace we obtain an inclusion 
$$ \cH_{\mu,l} \otimes \cH_{{}^t \mu,r} (0) =  
V_{\mu} \otimes V_{{}^t \mu} \hookrightarrow  \cH_{\lambda,1}(n_Y) $$
for some positive integer $n_Y$. It follows from the Skew Cauchy Formula \eqref{SkewCauchy} that
$$Y \in \YY_{r,l}^{fin}(\lambda) \qquad \iff \qquad n_Y = 0.$$

\bigskip

\section{The projective WZW-connection}

\subsection{Definition of the projective WZW-connection}

We now outline the definition of the projective WZW-connection on the sheaf
$\cV^\dagger_{\vec{\lambda},l}(\cF,r)$ over the smooth locus $\cB^s \subset \cB$
parameterizing smooth curves and refer to \cite{TUY} or \cite{U} for a detailed
account. Let $\cD \subset \cB$ be the discriminant locus and let $\cS = \coprod_{i=1}^n
s_i(\cB)$ be the union of the images of the $n$ sections. We recall the exact
sequence
\begin{equation} \label{deftheta}
0  \lra \pi_* \Theta_{\cC/ \cB}(* \cS) \lra \pi_* \Theta'_\cC(* \cS)_\pi
\stackrel{\theta}{\lra} \Theta_\cB(- \mathrm{log} \ \cD) \lra 0,
\end{equation}
where $\Theta_{\cC/\cB}(* \cS)$ denotes the sheaf of vertical rational vector fields on
$\cC$ with poles only along the divisor $\cS$, and $\Theta'_\cC(*\cS)_\pi$ the sheaf
of rational vector fields on $\cC$ with poles only along the divisor $\cS$ and with constant
horizontal components along the fibers of $\pi$. There is an $\cO_\cB$-linear map
$$ p : \pi_* \Theta'_\cC (* \cS)_\pi \lra \bigoplus_{i=1}^n \cO_\cB((\xi_i))
\frac{d}{d\xi_i},$$
which associates to a vector field $\vec{\ell}$ in $\Theta'_\cC (* \cS)_\pi$ the $n$
Laurent expansions $\ell_i \frac{d}{d\xi_i}$ around the divisor $s_i(\cB)$. Abusing notation
we also write $\vec{\ell}$ for its image under $p$
$$ \vec{\ell} = (\ell_1 \frac{d}{d\xi_1}, \cdots , \ell_n \frac{d}{d\xi_n} ) \in
\bigoplus_{i=1}^n \cO_\cB((\xi_i))  \frac{d}{d\xi_i}.$$
We then define for any vector field $\vec{\ell}$ in $\Theta'_\cC (* \cS)_\pi$ the
endomorphism $D(\vec{\ell})$ of $\cO_{\cB} \otimes \cH^\dagger_{\vec{\lambda},l}$ by
$$ D(\vec{\ell}) (f\otimes u) = \theta(\vec{\ell}) . f \otimes u +
\sum_{i=1}^n f \otimes (T[\ell_i] . u) $$
for $f$ a local section of $\cO_{\cB}$ and
$u \in \cH^\dagger_{\vec{\lambda},l}$. Here $T[\ell_i]$ denotes the action of the
energy-momentum tensor on the $i$-th component $\cH^\dagger_{\lambda_i,l}$, i.e., expanding
$\ell_i = \sum_{j= -n_0}^\infty \alpha_j \xi_i^{j+1}$ with $\alpha_j$ local sections of
$\cO_{\cB}$, the operator $T[\ell_i]$ equals $\sum_{j= -n_0}^\infty \alpha_j L_j$. Here the 
operators $L_j$ are the Virasoro operator acting 
linearly on $\cH^\dagger_{\lambda_i,l}$ via the Sugawara representation. It is shown
in \cite{TUY} that $D(\vec{\ell})$ preserves $\cV^\dagger_{\vec{\lambda},l}(\cF,r)$ and
that $D(\vec{\ell})$ only depends on the image $\theta(\vec{\ell})$ up to
homothety. One therefore obtains a projective connection $\nabla$ on the sheaf
$\cV^\dagger_{\vec{\lambda},l}(\cF,r)$ over $\cB^s$ given by
\begin{equation} \label{defconnection}
 \nabla_{\theta(\vec{\ell})} = \theta(\vec{\ell}) + T[\vec{\ell}].
\end{equation}

\bigskip

\subsection{Conformal embedding and projective flatness}

\begin{prop}
Let $\cF$ be a family of {\em smooth} projective $n$-pointed curves as in \eqref{definitionfamilyF}
and let $\vec{\lambda} \in P_1(rl)^n$ be a labelling of the marked points with fundamental 
weights of $\slfr(rl)$. Then for every collection of Young diagrams $\vec{Y} = (Y_1, \ldots, Y_n) \in 
\prod_{i=1}^n \YY_{r,l}^{aff}(\lambda_i)$ the tensor product of the $n$ inclusions
\eqref{inclH}
$$ \cH_{\vec{\mu},l} \otimes \cH_{{}^t \vec{\mu},r} \hookrightarrow  \cH_{\vec{\lambda},1}, $$
induces a natural homomorphism between sheaves of conformal blocks 
\begin{equation} \label{sdcb}
 (\cV^\dagger_{\vec{\lambda},1} (\cF, rl) , \nabla ) \lra
 (\cV^\dagger_{\vec{\mu},l} ( \cF , r) , \nabla )  \otimes 
 ( \cV^\dagger_{{}^t \vec{\mu},r} ( \cF, l), \nabla ),   
 \end{equation}
which is projectively flat for the WZW connections.
\end{prop}

\begin{proof}
We consider the embedding of semi-simple Lie algebras $\mathfrak{p} = \slfr(r) \times \slfr(l)
\subset \mathfrak{g} = \slfr(rl)$. Since this embedding is conformal, we have by \cite{KW}
Proposition 3.2(c) that for any integer $n$ the two Virasoro operators $L_n$ associated to
$\mathfrak{p}$ and $\mathfrak{g}$ coincide. The proposition now follows since the two linear
parts of the connections \eqref{defconnection} given by the energy-momentum tensor also 
coincide.
\end{proof}

\bigskip

\begin{rem}
{\em Note that the WZW-connection is only defined for a family of {\em smooth} curves.}
\end{rem}

\begin{rem}
{\em The previous proposition actually holds for any pair $\mathfrak{p} \subset \mathfrak{g}$ of 
conformal embeddings of semi-simple Lie algebras. For the list of conformal embeddings, 
see e.g. \cite {BB}.} 
\end{rem}

We have the following 

\begin{cor} \label{rankconstant}
With the above notation the homomorphism obtained from \eqref{sdcb}
$$ SD_{\vec{Y}} : (\cV_{\vec{\mu},l} (\cF,r), \nabla) \lra
(\cV_{\vec{\lambda},1} (\cF, rl) , \nabla )  \otimes 
( \cV^\dagger_{{}^t \vec{\mu},r} (\cF, l), \nabla ) $$
is projectively flat and has constant rank.
\end{cor}

\begin{proof}
Both assertions are valid for any vector bundles equipped with projective connections. We refer e.g. to
\cite{Be2} Lemma A.1 and Lemma A.2 for a proof.
\end{proof}

\bigskip

\section{Proof of the Main Theorem}

We prove the Main Theorem by induction on the genus $g$ of the curve. For $g=0$ the Theorem coincides
with \cite{NT} Theorem 4.4, since in that case $\dim \cV_{\vec{\lambda},1}(C, rl) = 1$ by Lemma
\ref{dimcblevelone}. Note that for $g=0$ the map $SD_{\vec{Y}}$ is an isomorphism.

\bigskip

We now assume that the Theorem holds for any smooth marked curve of genus $g-1$. First of all we notice that
the map $SD_{\vec{Y}}$ as defined in the Main Theorem for smooth curves can be defined as well for a nodal
curve. Next, we observe that it is enough to show injectivity of the map $SD_{\vec{Y}}$ for a curve $C$
of genus $g$ with one node. In fact, by upper semi-continuity of the rank of a homomorphism between
vector bundles, we then obtain injectivity of $SD_{\vec{Y}}$ for a general smooth curve. Then, since the rank
of $SD_{\vec{Y}}$ is constant for smooth families, as shown in Corollary \ref{rankconstant}, we obtain injectivity
for any smooth curve.

\bigskip

We consider the desingularization $\pi : \widetilde{C} \ra C$ of the nodal curve $C$ of genus $g$. Note that
$\widetilde{C}$ is smooth of genus $g-1$. The linear map $\alpha$ decomposes under the factorization
given by Proposition \ref{factorization} as follows

$$\begin{CD}
\cV^\dagger_{\vec{\lambda},1}(C, rl) @>\alpha>>  \cV^\dagger_{\vec{\mu},l} (C, r) \otimes 
\cV^\dagger_{{}^t \vec{\mu},r} (C, l) \\
@VV\cong V   @VV \cong V \\
\bigoplus_{\lambda_0 \in P_1(rl)} \cV^\dagger_{\vec{\lambda} \cup \lambda_0 \cup
\lambda_0^\dagger,1}(\widetilde{C}, rl) @>>> \bigoplus_{(\mu_1, \mu_2) \in P_l(r) \times P_r(l)}
\cV^\dagger_{\vec{\mu} \cup \mu_1 \cup \mu_1^\dagger ,l} (\widetilde{C}, r) \otimes 
\cV^\dagger_{{}^t \vec{\mu} \cup \mu_2 \cup \mu_2^\dagger,r} (\widetilde{C}, l)
\end{CD}
$$
For any triple $(\lambda_0, \mu_1, \mu_2) \in P_1(rl) \times P_l(r) \times P_r(l)$ we define the map
$$ \alpha_{\lambda_0, \mu_1, \mu_2} : \cV^\dagger_{\vec{\lambda} \cup \lambda_0 \cup
\lambda_0^\dagger,1}(\widetilde{C}, rl) \lra 
\cV^\dagger_{\vec{\mu} \cup \mu_1 \cup \mu_1^\dagger ,l} (\widetilde{C}, r) \otimes 
\cV^\dagger_{{}^t \vec{\mu} \cup \mu_2 \cup \mu_2^\dagger,r} (\widetilde{C}, l) $$
as the $(\lambda_0, \mu_1, \mu_2)$ component of $\alpha$ in the above decomposition.

\bigskip

\begin{defi}  We say that a triple $(\lambda_0, \mu_1, \mu_2) \in P_1(rl) \times P_l(r) \times P_r(l)$ 
is {\em admissible} if there exists a Young diagram $Y \in \YY_{r,l}^{aff}(\lambda_0)$
such that $\pi(Y) = \mu_1 \in P_l(r)$ and $\pi({}^t Y) = \mu_2 \in P_r(l)$. 
\end{defi}

It is clear from the definition that if $(\lambda_0, \mu_1, \mu_2)$ is admissible then 
the corresponding Young diagram $Y \in \YY_{r,l}^{aff}(\lambda_0)$ is uniquely determined either by 
the pair $(\lambda_0, \mu_1)$ or by the pair $(\lambda_0, \mu_2)$.

\bigskip

We now recall \cite{BP} Proposition 4.4 in our context, i.e., for the conformal embedding
$\mathfrak{p} = \slfr(r) \times \slfr(l) \subset \mathfrak{g} = \slfr(rl)$.

\begin{prop} \label{factconf}
Given a triple $(\lambda_0, \mu_1, \mu_2) \in P_1(rl) \times P_l(r) \times P_r(l)$ the linear
map $\alpha_{\lambda_0, \mu_1, \mu_2}$ is
\begin{itemize}
\item identically zero, if $(\lambda_0, \mu_1, \mu_2)$ is not admissible, or if  
 $(\lambda_0, \mu_1, \mu_2)$ is admissible and $Y \in \YY_{r,l}^{aff}(\lambda_0) \setminus
\YY_{r,l}^{fin}(\lambda_0)$
\item is induced by the natural inclusion
$$\cH_{\vec{\mu},l} \otimes \cH_{{}^t\vec{\mu},r} \otimes \cH_{\mu_1,l} \otimes \cH_{\mu_2,r} \otimes 
\cH_{\mu_1^\dagger,l} \otimes \cH_{\mu_2^\dagger,r} \hookrightarrow \cH_{\vec{\lambda},1} \otimes 
\cH_{\lambda_0,1}\otimes \cH_{\lambda_0^\dagger,1}, $$
if  $(\lambda_0, \mu_1, \mu_2)$ is admissible and $Y \in \YY_{r,l}^{fin}(\lambda_0)$.
\end{itemize}
\end{prop}

\bigskip

We denote by 
$$SD(\lambda_0, \mu_1, \mu_2) : \cV_{\vec{\mu} \cup \mu_1 \cup \mu_1^\dagger ,l}  (\widetilde{C}, r) \lra
\cV_{\vec{\lambda} \cup \lambda_0 \cup \lambda_0^\dagger,1}(\widetilde{C}, rl) \otimes 
\cV^\dagger_{{}^t \vec{\mu} \cup \mu_2 \cup \mu_2^\dagger,r} (\widetilde{C}, l) $$
the linear map induced by $\alpha_{\lambda_0, \mu_1, \mu_2}$. If we are in the second case of
Proposition \ref{factconf}, i.e., if 
$(\lambda_0, \mu_1, \mu_2)$ is admissible and $Y \in \YY_{r,l}^{fin}(\lambda_0)$, then Proposition 
\ref{factconf} and the induction hypothesis applied to the curve $\widetilde{C}$
implies that the map $SD(\lambda_0, \mu_1, \mu_2)$ is injective. In all other cases this
map is zero. Now we observe that for any pair $(\lambda_0, \mu_2) \in P_1(rl) \times P_r(l)$
there exists at most one $\mu_1 \in P_l(r)$ such that the triple $(\lambda_0, \mu_1, \mu_2)$
is admissible. We introduce the set $S(\mu_1) = \{ (\lambda_0, \mu_2) \ |  \  (\lambda_0, \mu_1, \mu_2) \ 
\text{admissible} \}$. In order to show injectivity of $SD_{\vec{Y}}$ it is enough to show 
for each $\mu_1 \in P_l(r)$ injectivity of the map
$$
\cV_{\vec{\mu} \cup \mu_1 \cup \mu_1^\dagger ,l}  (\widetilde{C}, r) \lra
\bigoplus_{(\lambda_0, \mu_2) \in S(\mu_1)}
\cV_{\vec{\lambda} \cup \lambda_0 \cup \lambda_0^\dagger,1}(\widetilde{C}, rl) \otimes 
\cV^\dagger_{{}^t \vec{\mu} \cup \mu_2 \cup \mu_2^\dagger,r} (\widetilde{C}, l).
$$
But this follows from the fact that the map  $SD(\lambda_0, \mu_1, \mu_2)$ is injective
for an admissible triple with $Y \in \YY_{r,l}^{fin}(\lambda_0)$. Since $\pi :
\YY_{r,l}^{fin} \ra P_l(r)$ is surjective, we can find for each $\mu_1 \in P_l(r)$ such 
a pair $(\lambda_0, \mu_2) \in S(\mu_1)$. This completes the proof.


\bigskip






\begin{thebibliography}{999999}

\bibitem[BB]{BB} A. Bais, P. Bouwknegt: A classification of subgroup truncations of the bosonic string, Nuclear Physics B279 (1987), 561-570


\bibitem[BNR]{BNR} A. Beauville, M.S. Narasimhan, S. Ramanan: Spectral curves and the generalised theta
divisor, J. Reine Angew. Math. 398 (1989), 169-179

\bibitem[Be1]{Be1} P. Belkale: The strange duality conjecture for generic curves. J. Amer.
Math. Soc. 21 (2008), 235-258

\bibitem[Be2]{Be2} P. Belkale: Strange duality and the Hitchin/WZW connection, J.
Differential Geom. 82 (2009), no. 2, 445-465

\bibitem[BP]{BP} A. Boysal, C. Pauly: Strange duality for Verlinde spaces of exceptional groups at level one, International Mathematics Research Notices (2010), 595-618

\bibitem[G]{G} D. Gepner: Fusion rings and Geometry, Commun. Math. Phys. 141 (1991), 381-411 

\bibitem[H]{H} K. Hasegawa: Spin module versions of Weyl’s reciprocity theorem for classical Kac-Moody Lie
algebras — an application to branching rule duality, Publ. Res. Inst. Math. Sci. 25 (1989), 741–828

\bibitem[KW]{KW} V. Kac, M. Wakimoto: Modular and conformal invariance constraints in representation theory
of affine algebras, Advances in Math. 70 (1988), 156-234

\bibitem[MO1]{MO} A. Marian, D. Oprea: The level-rank duality for non-abelian theta functions, Invent. Math. 168, 
no.2 (2007), 225-247

\bibitem[MO2]{MO2} A. Marian, D. Oprea: A tour of theta dualities on moduli spaces of sheaves, 
Curves and abelian varieties, 175-202, Contemporary Mathematics, 465, American Mathematical Society, Providence, Rhode Island (2008)

\bibitem[NT]{NT} T. Nakanishi, A. Tsuchiya: Level-Rank Duality of WZW Models in Conformal Field Theory, 
Commun. Math. Phys. 144 (1992), 351-372

\bibitem[O]{O} R. Oudompheng: Rank-level duality for conformal blocks of the linear group, J. Algebraic Geometry
20 (2011), 559-597


\bibitem[Pa1]{Pa1} C. Pauly: Espaces de modules de fibr\'es paraboliques et blocs conformes, Duke Math. Journal,
Vol. 86, No. 1 (1996), 217-235


\bibitem[Pa2]{Pa2} C. Pauly: La dualit\'e  \'etrange, S\'eminaire Bourbaki, Vol. 2007/2008, Ast\'erisque No. 326 (2009), Exp. No. 994, 363-377

\bibitem[Po]{Po} M. Popa: Generalized theta linear series on moduli spaces of vector bundles on curves, 
   to appear in the Handbook of Moduli, G. Farkas and I. Morrison eds.
      
\bibitem[Pr]{Pr} C. Procesi: Lie Groups. An Approach through Invariants and Representations. Universitext (2007),
Springer

\bibitem[TUY]{TUY} A. Tsuchiya, K. Ueno, Y. Yamada: Conformal Field Theory on Universal Family of Stable Curves
with Gauge Symmetries, Advanced Studies in Pure Mathematics 19 (1989), Kinokuniya Shoten and Academic Press,
459-566

\bibitem[U]{U} K. Ueno: Introduction to conformal field theory with gauge symmetries, in Geometry and Physics, Lecture Notes in Pure and Applied Mathematics 184, Marcel Dekker, 1996, 603-745


\end{thebibliography}
\end{document}